\pdfoutput=1 %arxiv required

\documentclass[12pt, letterpaper, leqno]{amsart}
\usepackage[letterpaper, margin=1in]{geometry}
\usepackage[utf8]{inputenc}
\usepackage[style=alphabetic, maxnames=99]{biblatex}
\AtBeginBibliography{\small}
\usepackage{microtype}
\microtypesetup{protrusion=false}
\usepackage[english]{babel}
\usepackage{amsfonts}
\usepackage{amsmath}
\usepackage{amsthm}
\usepackage{amssymb}
\usepackage{mathrsfs}
\usepackage{xpatch}
\usepackage{stmaryrd}
\usepackage{tikz}
\usepackage{tikz-cd}
\usetikzlibrary{arrows.meta,automata,positioning}
\usetikzlibrary{svg.path}
\usepackage{amscd}
\usepackage{quiver}
\usepackage{tikzit}
\usepackage{float}
\usepackage{enumitem}
\usepackage[labelfont=bf]{caption}
\usepackage[linewidth=0pt]{mdframed}
\usepackage{subfig}
\usepackage{listings}
\usepackage{fancyhdr}
\usepackage[bottom, perpage]{footmisc}
\usepackage{etoolbox}

\patchcmd{\section}{\scshape}{\bfseries}{}{}
\makeatletter
\renewcommand{\@secnumfont}{\bfseries}
\makeatother

\allowdisplaybreaks
\usepackage{subfiles}
\usepackage{setspace}
\numberwithin{equation}{section}
\numberwithin{figure}{section}
\makeatletter
\let\c@equation\c@figure
\makeatother
\usepackage{xr-hyper}
\usepackage{xstring}
\usepackage{hyperref}
\usepackage{url}
\hypersetup{
	colorlinks   = true, 
	urlcolor     = cyan, 
	linkcolor    = blue, 
	citecolor   = blue 
}

\makeatletter
\newcommand{\labeltext}[3][]{%
	\@bsphack%
	\csname phantomsection\endcsname% in case hyperref is used
	\def\tst{#1}%
	\def\labelmarkup{}% How to markup the label itself
	\def\refmarkup{}%
	\ifx\tst\empty\def\@currentlabel{\refmarkup{#2}}{\label{#3}}%
	\else\def\@currentlabel{\refmarkup{#1}}{\label{#3}}\fi%
	\@esphack%
	\labelmarkup{#2}% visible printed text.
}
\makeatother
\newtheoremstyle{mydefinition}{10pt}{10pt}{}{}{\bfseries}{.}{.5em}{}
\theoremstyle{definition}
\newtheorem{definition}[equation]{Definition}
\newtheorem*{definition*}{Definition}

\newtheorem*{notation*}{Notation}
\newtheorem{remark}[equation]{Remark}
\newtheorem*{remark*}{Remark}

\newtheorem*{examplexnonum}{Example}

\newtheoremstyle{mytheorem}{10pt}{10pt}{\itshape}{}{\bfseries}{.}{.5em}{}
\theoremstyle{theorem}
\newtheorem{theorem}[equation]{Theorem}
\newtheorem*{theorem*}{Theorem}
\newtheorem{proposition}[equation]{Proposition}
\newtheorem*{proposition*}{Proposition}
\newtheorem{lemma}[equation]{Lemma}
\newtheorem*{lemma*}{Lemma}

\newtheorem*{corollary*}{Corollary}
\newtheorem{conjecture}[equation]{Conjecture}
\newtheorem*{conjecture*}{Conjecture}

\theoremstyle{definition}

\newtheoremstyle{mynamed}{}{}{\itshape}{}{\bfseries}{.}{.5em}{\thmnote{#3}\thmnumber{. #2}}
\theoremstyle{mynamed}

\newenvironment{example*}{
	\pushQED{\qed}\examplexnonum
}{
	\popQED\endexamplexnonum
}

\xpatchcmd{\proof}{\itshape}{\normalfont\proofnamefont}{}{}
\newcommand{\proofnamefont}{\bfseries}
\usepackage{mystyle}

\newcommand{\define}[1]{\emph{#1}}

\newcommand{\RR}{\mathbb{R}}
\newcommand{\QQ}{\mathbb{Q}}

\newif\ifhascomments \hascommentstrue
\ifhascomments
\newcommand{\matt}[1]{{\color{red}[[\ensuremath{\spadesuit\spadesuit\spadesuit} #1]]}}
\newcommand{\sean}[1]{{\color{red}[[\ensuremath{\clubsuit\clubsuit\clubsuit} #1]]}}
\else
\newcommand{\matt}[1]{}
\newcommand{\sean}[1]{}
\fi

\renewcommand{\setminus}{\smallsetminus}
\title{Approximating rational points on horospherical varieties}
\author{Sean Monahan and Matthew Satriano}

\address{Sean Monahan, Department of Pure Mathematics, University of Waterloo}
\email{sean.monahan@uwaterloo.ca}

\address{Matthew Satriano, Department of Pure Mathematics, University of Waterloo}
\email{msatrian@uwaterloo.ca}

\begin{document}
	
\begin{abstract}
Let $X$ be a smooth projective split horospherical variety over a number field $k$ and $x\in X(k)$. Contingent on Vojta's conjecture, we construct a curve $C$ through $x$ such that (in a precise sense) rational points on $C$ approximate $x$ better than any Zariski dense sequence of rational points. This proves a weakening of a conjecture of McKinnon in the horospherical case. Our results make use of the minimal model program and apply as well to $\Q$-factorial horospherical varieties with terminal singularities.
\end{abstract}
	
\maketitle
\setcounter{tocdepth}{1}
\tableofcontents

\section{Introduction}\label{sec:introduction}

In 1842, Dirichlet proved his famous Approximation Theorem, stating that for every irrational number $x$, there are infinitely many rational numbers $\frac{a}{b}$ for which $|x-\frac{a}{b}|<\frac{1}{b^2}$. One may rephrase the Approximation Theorem slightly differently as follows. Let the approximation exponent $\tau_x$ be the unique element of $(0,\infty]$ such that the inequality
$$\left|{x-\frac{a}{b}}\right| \leq \frac{1}{b^{\tau_{x}+\delta}}$$
has only finitely many solutions $\frac{a}{b}\in \QQ$ in reduced form whenever $\delta>0$, and has infinitely many solutions when $\delta<0$. Then Dirichlet's Theorem states that $\tau_x\geq2$ for all irrational $x$. In 1844, Liouville \cite{liouville1844nouvelle} proved that if $x\in \RR$ is algebraic of degree $d$ over $\QQ$, then $\tau_{x}\leq d$. This upper bound was subsequently improved by Thue \cite{thue1909uber}, Siegel \cite{siegel1921approximation}, Dyson \cite{dyson1947approximation}, and Gelfand, leading finally to Roth's crowning theorem \cite{roth1955rational} that $\tau_x\leq2$ for all algebraic $x\in\RR$. Therefore, Dirichlet's Theorem and Roth's Theorem together show that $\tau_x=2$ for all irrational and algebraic $x$.

In \cite{mckinnon2015seshadri}, McKinnon and Roth generalized $\tau_x$ (or rather its reciprocal) to arbitrary projective varieties $X$ over a number field $k$ by replacing the function $|x-\frac{a}{b}|$ by a distance function $\dist_v(x,\cdot)$ depending on a place $v$ of $k$, and replacing $\frac{1}{b}$ by a height function $H_D(\cdot)$ depending on an ample divisor $D$. %An essential change, however, is that they moved the exponent $\tau_x$ from the height to the distance; this was done to make their generalized exponents behave better with respect to changes in $D$. 
Given a sequence $\{x_i\}\subseteq X(k)$ approximating a point $p$, one then obtains an associated \emph{approximation constant} $\alpha_{p,\{x_i\}}(D)$; see Section \ref{sec:defining the approximation constant} for the precise definition. The constant $\alpha_p(D)$ is defined to be the infimum of $\alpha_{p,\{x_i\}}(D)$ over all choices of sequences $\{x_i\}$; if one restricts attention only to sequences contained in a subvariety $Z\subseteq X$, then the resulting infimum is denoted by $\alpha_{p,Z}(D)$.

In \cite{mckinnon2007conjecture}, McKinnon conjectured that any sequence of best approximation (i.e., a sequence of rational points which minimizes $\alpha_{p,\{x_i\}}(D)$) must, in fact, lie on a rational curve.

\begin{conjecture}[{\cite[Conjecture 2.7]{mckinnon2007conjecture}}]\label{conj:mckinnon}
	Let $X$ be a projective variety defined over a number field $k$, let $D$ be an ample divisor on $X$, and fix a point $p\in X(k)$ for which there is a rational curve, defined over $k$, passing through $p$. Then there exists a curve $C\subseteq X$ (necessarily rational) such that $\alpha_{p,C}(D)=\alpha_p(D)$. 
\end{conjecture}

This conjecture is known in some special cases, primarily in dimension $2$: it was shown for split rational surfaces of Picard rank at most four in \cite{mckinnon2007conjecture}, cubic surfaces in \cite{mckinnon2016analogue}, and blow-ups of the $n$-th Hirzebruch surface at special configurations of at most $2n$ points in \cite{santos2019rational}. The conjecture was verified for special classes of toric varieties in \cite{huang2021rational}, namely when $p$ is in the open torus and the pseudo-effective cone $\overline{\mathrm{Eff}}(X)$ is simplicial. This conjecture was also recently proved by McKinnon and the second-named author \cite{satriano2021approximating} for smooth projective split toric surfaces $X$ assuming Vojta's Main Conjecture \cite[Conjecture 3.4.3]{vojta1987diophantine}; higher dimensional analogues of the conjecture were also shown \cite{satriano2021approximating}. Our results in this current paper generalize those of \cite{satriano2021approximating} to the much broader class of horospherical varieties.

Recall that a \define{horospherical variety} is a normal variety $X$ equipped with the action of a connected reductive algebraic group $G$ in such a way that $X$ has an open orbit whose points are stabilized by a maximal unipotent subgroup of $G$. When $G=T$ is a torus, we recover the class of toric varieties. As with toric varieties and fans, horospherical varieties have a rich combinatorial theory using so-called \textit{coloured fans}; see \cite{luna1983plongements,knop1991luna,monahan2023overview} for information on this topic. The advantage to working with horospherical geometry is that it is vastly broader than toric geometry, e.g. the horospherical class contains the class of all flag varieties and includes many examples of projective varieties with non-binomial relations. 

Horospherical varieties have seen much activity. The (log) minimal model program for horospherical varieties was determined in \cite{brion1993mori,pasquier2018log}, Fano horospherical varieties were classified in \cite{pasquier2008varietes}, smooth horospherical varieties with Picard number 1 were classified in \cite{pasquier2009smooth}, and the Cox ring and Cox construction (as introduced in \cite{cox1995homogeneous}) was describe in \cite{gagliardi2014cox,gagliardi2019luna}. %,monahan2023horospherical}. 
In these results, the aforementioned combinatorial theory plays an indispensable role in understanding horospherical varieties. 
%on this combinatorial theory to prove the main results in this paper. 

%The aforementioned combinatorial theory plays an indispensable role in understanding horospherical varieties, e.g. it has been used to describe 

%Currently, Conjecture \ref{conj:mckinnon} is largely out of reach without imposing additional constraints. In higher dimensions, one can consider a weaker form of the problem, where one aims to prove that sequences of points on rational curves approximate $p$ better than any Zariski dense sequence converging to $p$; in the case when $X$ is a surface, this weaker problem is equivalent to Conjecture \ref{conj:mckinnon}.

To state our results precisely, we recall the following definition %for a special class of rational points 
introduced in \cite{satriano2021approximating}. 

\begin{definition}\label{def:canonically bounded}
	Let $X$ be a projective $\Q$-Gorenstein variety defined over a number field $k$, and let $p\in X(k)$. We say that $X$ is \define{canonically bounded at $p$} if $\alpha_{p,\{x_i\}}(-K_X)\geq \dim(X)$ for all Zariski dense sequences $\{x_i\}\subseteq X(k)$. 
\end{definition}

Under the presence of Vojta's Main Conjecture, canonical boundedness is automatic for smooth horospherical varieties, as we state below. Note that, by \define{split} horospherical variety over a number field $k$, we mean that the acting group $G$ is \textit{split} reductive over $k$, i.e. $G$ is defined over $k$ and has a split maximal torus over $k$. 

\begin{proposition}\label{prop:Vojta->canboudned}
	Let $X$ be a smooth projective split horospherical variety over a number field $k$. Then Vojta's Main Conjecture for $X$ implies that all $p\in X(k)$ are canonically bounded.
\end{proposition}

With the assumption of canonical boundedness, we are able to prove the weaker version of Conjecture \ref{conj:mckinnon} for projective split horospherical varieties; this is the main result of our paper, stated below. %Note that, in order to remove the $\Q$-Gorenstein assumption on our varieties, 
We say a point is \textit{extrinsically} canonically bounded if it is canonically bounded in a partial resolution $X'\to X$ where $X'$ has terminal singularities (see Definition \ref{def:extrinsically canonically bounded}). 

\begin{theorem}\label{thm:main theorem}
	Let $X$ be a projective split horospherical variety over a number field $k$, and let $p\in X(k)$. If $X$ is extrinsically canonically bounded at $p$, then for all $\Q$-Cartier nef $\Q$-divisors $D$ on $X$, there is an irreducible rational curve $C$ through $p$ such that $C$ is unibranch at $p$ and 
	\begin{align*}
		\alpha_{p,C}(D) \leq \alpha_{p,\{x_i\}}(D)
	\end{align*}
	for all Zariski dense sequences $\{x_i\}\subseteq X(k)$.
\end{theorem}

\begin{remark}\label{rmk:ample divisors vs nef in conjecture}
	Note that while Conjecture \ref{conj:mckinnon} is only stated for ample divisors, our Theorem \ref{thm:main theorem} applies to the more general class of $\Q$-Cartier nef $\Q$-divisors.
\end{remark}

Assuming Vojta's Main Conjecture, Proposition \ref{prop:Vojta->canboudned} tells us that we may replace the canonically bounded assumption on $p$ in Theorem \ref{thm:main theorem} with the simpler assumption that $p$ be smooth. Therefore, we obtain the following result. 

\begin{theorem}\label{thm:consequence of main theorem and Vojta}
	Let $X$ be a projective split horospherical variety over a number field $k$. Assume that Vojta's Main Conjecture holds for some projective horospherical (strong) resolution of singularities of $X$. Then, for all smooth points $p\in X(k)$ and all $\Q$-Cartier nef $\Q$-divisors $D$ on $X$, there exists an irreducible rational curve $C$ through $p$ such that $C$ is unibranch at $p$ and 
	\begin{align*}
		\alpha_{p,C}(D) \leq \alpha_{p,\{x_i\}}(D)
	\end{align*}
	for all Zariski dense sequences $\{x_i\}\subseteq X(k)$. 
\end{theorem}

%Lastly, in Section \ref{sec:review of horospherical MMP}, we provide a brief review of the horospherical Minimal Model Program. We hope this will serve as a useful reference.

\subsection*{Notation and conventions}

Throughout this paper we use the following notation. All varieties are assumed to be irreducible, and all varieties and algebraic groups are assumed to be defined over a number field $k$. \emph{Moreover, we implicitly assume that all horospherical varieties are split.}

Regarding horospherical varieties and coloured fans, we use the notation and conventions from \cite{monahan2023overview}. Throughout the paper, we fix a connected split reductive (linear) algebraic group $G$, a Borel subgroup $B\subseteq G$, a maximal torus $T\subseteq B$, and let $U\subseteq B$ denote the maximal unipotent subgroup of $B$. We may assume that our fixed torus $T$ is split over $k$. Fix a horospherical subgroup $H\subseteq G$ containing $U$, and let $P:=N_G(H)$ denote the associated parabolic subgroup. Let $N=N(G/H)$ denote the coloured lattice associated to $G/H$, which has universal colour set $\calC=\calC(G/H)$ and colour points $\{u_\alpha\in N:\alpha\in\calC\}$. 

For a horospherical $G/H$-variety $X$ and $\alpha\in\calC$, we let $D_\alpha^X$ denote the colour divisor in $X$ corresponding to $\alpha$, or we simply write $D_\alpha$ when there is no confusion. Let $\calD(X)$ denote the set of $B^-$-invariant prime divisors in $X$, which consists of the colour divisors together with the $G$-invariant prime divisors. Let $\Sigma^c(X)$ denote the coloured fan associated to $X$, let $\calF(\Sigma^c(X))=\calF(X)\subseteq\calC$ denote its colour set, and let $\Sigma^c(X)(1)$ denote the set of coloured rays of $\Sigma^c(X)$. 

Note that the group $\Aut^G(X)$ of $G$-equivariant automorphisms of $X$ is a torus whose lattice of one-parameter subgroups is $N$. Since $\Aut^G(X)$ is a quotient of $T$, it is a split torus over $k$.

\subsection*{Acknowledgements}

We thank Changho Han and David McKinnon for helpful conversations. The first-named author was partially supported by a PGS-D scholarship from the National Sciences and Engineering Research Council of Canada (reference number: PGSD3-558713-2021). The second-named author was partially supported by a Discovery Grant from the National Sciences and Engineering Research Council of Canada as well as a Mathematics Faculty Research Chair from the University of Waterloo.

\section{Defining the approximation constant}\label{sec:defining the approximation constant}

In this section, we provide the definition of the approximation constant. For a more detailed discussion, see \cite{mckinnon2015seshadri}. Recall that all varieties are defined over a number field $k$. 

\begin{definition}\label{def:approx constant for sequence}
	Let $X$ be a projective variety, $v$ be a place of $k$, $p\in X(\overline{k})$, and $D$ be a $\QQ$-Cartier $\Q$-divisor
	%where height is interpreted multiple to get Cartier then root height
	on $X$.  For any sequence $\{x_i\}\subseteq X(k)$ of distinct
	points with $\dist_v(p,x_i)\rightarrow 0$, we set
	\begin{align*}
		A(\{x_i\}, D) := \left\{{
			\gamma\in\RR\mid
			\dist_v(p,x_i)^{\gamma} H_D(x_i)\,\,\mbox{is bounded from above}
		}\right\}.
	\end{align*}
	We define the \define{approximation constant} of $\{x_i\}$ with respect to $D$ to be
	\begin{align*}
		\alpha_{p,\{x_i\}}(D)=
		\begin{cases}
			\inf A(\{x_i\},D),& A(\{x_i\},D)\neq\varnothing\\
			\infty,& A(\{x_i\},D)=\varnothing
		\end{cases}.
	\end{align*} 
\end{definition}

\begin{remark}\label{rmk:approx constant and subsequences}
	Note that, if $\{x_i'\}$ is a subsequence of $\{x_i\}$, then $A(\{x_i\},D)\subseteq A(\{x_i'\},D)$.  Thus, $\alpha_{p,\{x_i'\}}(D)\leq \alpha_{p,\{x_i\}}(D)$, so we may freely replace a sequence with a subsequence when establishing lower bounds on $\alpha_{p,\{x_i\}}(D)$.
\end{remark}

As $i\to\infty$ we have $\dist_v(p,x_i)\to0$.  So, one expects that $\dist_v(p,x_i)^{\gamma}H_{D}(x_i)$ tends to $0$ for large $\gamma$ and to $\infty$ for small $\gamma$.  The number $\alpha_{p,\{x_i\}}(D)$ marks the transition point between these two behaviours.

\begin{definition}\label{def:approx constant}
	Let $X$ be a projective variety, $D$ a $\QQ$-Cartier divisor on $X$, and $p\in X(\overline{k})$.  Then $\alpha_{p}(D)$ is defined to be the infimum of all $\alpha_{p,\{x_i\}}(D)$ as we range over sequences of distinct points $\{x_i\}\subseteq X(k)$ converging to $p$.  If no such sequence exists, then set $\alpha_{p}(D)=\infty$.
\end{definition}

To finish this section, we include the definition of \textit{extrinsically} canonically bounded. Recall that terminal singularities are the natural class one considers when running the minimal model program (MMP).

\begin{definition}\label{def:extrinsically canonically bounded}
	Suppose that $X$ is a projective (split) horospherical $G$-variety, and let $p\in X(k)$. We say that $X$ is \define{extrinsically canonically bounded at $p$} if there exists a horospherical birational map $f:X'\to X$ where $X'$ is a projective terminal $\Q$-factorial horospherical $G$-variety, $f$ is an isomorphism at $p$, and $X'$ is canonically bounded at $p$ (in the sense of Definition \ref{def:canonically bounded}).
\end{definition}

%\begin{remark}\label{rmk:comments on extrinsically canonically bounded}
%	We make two remarks on Definition \ref{def:extrinsically canonically bounded}. First, note that such a resolution $f$ \mynote{always exists} for $X$. Second, if $X$ is already $\Q$-factorial and terminal, then $X$ being extrinsically canonically bounded at $p$ is \mynote{equivalent} to $X$ being canonically bounded at $p$. 
%\end{remark}

\section{Review of horospherical MMP}\label{sec:review of horospherical MMP}

We provide a brief review some aspects of the minimal model program (MMP) for horospherical varieties. % since the strategy for the proof of Theorem \ref{thm:main theorem} relies on the MMP. \sean{review this following sentence} 
Much of this section involves specializing the notation and results of Brion's paper \cite{brion1993mori} on spherical MMP to the horospherical subclass, which allows us to properly formulate the combinatorial arguments that are essential to proving Theorem \ref{thm:main theorem}. 

Throughout this section, let $X$ be a projective $\Q$-factorial horospherical $G/H$-variety. Denote its coloured fan by $\Sigma^c$, which is on the coloured lattice $N$. Since $\Aut^G(X)$ is split, $N$ and, by extension, $\Sigma^c$ are $\Gal(\ol{k}/k)$-invariant, which implies that $\Pic(X)$ is $\Gal(\ol{k}/k)$-invariant, so all MMP contractions of $X$ are well-defined over $k$. 

Let $\cont_\calR:X\to X_\calR$ denote the contraction of an extremal ray $\calR$ of the Mori cone of $X$. From \cite[Theorem 3.1]{brion1993mori} we know that $X_\calR$ is also a projective horospherical $G$-variety, and the map is a horospherical morphism. Thus, $X_\calR$ corresponds to a coloured fan $\Sigma_\calR^c$ on a coloured lattice $N_\calR$. Recall that there are three possibilities for $\cont_\calR$: it can be a Mori fibre space, a divisorial contraction, or a flipping contraction. In the last case, $X_\calR$ is not $\Q$-factorial, but we have a flip $\psi:X\dashrightarrow X_+$ where $X_+$ is a projective $\Q$-factorial horospherical $G/H$-variety; see \cite[Proposition 4.3]{brion1993mori}. Let $\Sigma_+^c$ denote the coloured fan corresponding to $X_+$, which lives on the same coloured lattice $N$.

\subsection{The Mori cone}\label{subsec:the Mori cone}

Following \cite[Section 3.2]{brion1993mori}, the Mori cone $\NE(X)$ is generated by a finite set of curve classes, which we describe as follows. Up to linear equivalence, every Cartier divisor on $X$ has the form $\delta=\sum_{(\rho,\varnothing)} a_\rho D_\rho + \sum_{\alpha\in\calC} a_\alpha D_\alpha$ for some $a_\rho,a_\alpha\in\Z$, where the first sum is over all non-coloured rays $(\rho,\varnothing)$ in $\Sigma^c$; see \cite[Section 7]{monahan2023overview} for a review of divisors on horospherical varieties. Fix such a $\delta$, and let $\varphi_\delta:N_\R\to\R$ denote its associated piecewise linear function with Cartier data $\{m_\sigma\}_{\sigma\in\Sigma_{\max}}\subseteq N^\vee$; so $\varphi_\delta(u_\alpha)=a_\alpha$ for each $\alpha\in\calC$ and $\varphi_\delta(v_\rho)=a_\rho$ for each primitive generator $v_\rho$ of a non-coloured ray $(\rho,\varnothing)$. Note that this notation also works for $\Q$-Cartier $\delta$ by allowing rational coefficients and rational Cartier data. 

Given a wall $\mu\in\Sigma$, write $\mu=\sigma_+\cap\sigma_-$ for unique maximal cones $\sigma_+,\sigma_-\in\Sigma_{\max}$. Choose $m^\mu\in N^\vee$ to be a primitive lattice element such that the functional $\langle m^\mu,-\rangle$ is positive on $\sigma^+$, zero on $\mu$, and negative on $\sigma_-$. Then $m_{\sigma_+}-m_{\sigma_-}\in N^\vee$ is an integer multiple of $m^\mu$, which we denote $\frac{m_{\sigma_+}-m_{\sigma_-}}{m^\mu}\in\Z$. Define $C_\mu\in N_1(X)=N^1(X)^\vee$ via
\begin{align}\label{eq:wall curve class}
	\delta \cdot C_\mu := \frac{m_{\sigma_+}-m_{\sigma_-}}{m^\mu}.
\end{align}
We call $C_\mu$ a \define{wall curve class}. Note that there exists an irreducible $B^-$-invariant curve (isomorphic to $\P^1$) whose class in $N_1(X)$ is $C_\mu$; see \cite[Proposition 3.3]{brion1993mori}. 

On the other hand, given a pair $(\alpha,\sigma^c)$ where $\sigma^c\in\Sigma^c_{\max}$ and $\alpha\in\calC\setminus\calF(\sigma^c)$, we define $C_{\alpha,\sigma^c}\in N_1(X)=N^1(X)^\vee$ via
\begin{align}\label{eq:colour curve class}
	\delta \cdot C_{\alpha,\sigma^c} := a_\alpha - \langle m_\sigma,u_\alpha\rangle.
\end{align}
We call $C_{\alpha,\sigma^c}$ a \define{colour curve class}. Note that, if the ray $\R_{\geq 0} C_{\alpha,\sigma^c}\subseteq N_1(X)_\R$ does not contain any wall curve classes, then there exists an irreducible $B^-$-invariant curve (isomorphic to $\P^1$) whose class in $N_1(X)$ is $C_{\alpha,\sigma^c}$; see \cite[Proposition 3.6]{brion1993mori}. 

Then $\NE(X)$ is generated by all these wall curve classes and colour curve classes, i.e.
\begin{align*}
	\NE(X) = \Cone\Bigg(C_\mu,C_{\alpha,\sigma^c} ~:~ \begin{split}&\mu\in\Sigma \text{ is a wall}, \\ &(\alpha,\sigma^c): ~\sigma^c\in\Sigma^c_{\max},\alpha\in\calC\setminus\calF(\sigma^c)\end{split} \Bigg) \subseteq N_1(X)_\R.
\end{align*}

By \cite[Theorem 3.2]{brion1993mori}, every extremal ray of $\NE(X)$ which does not contain any wall curve classes is generated by a colour curve class $C_{\alpha,\sigma^c}$ with $\alpha\notin\calF(X)$ and $u_\alpha\in\sigma$. Therefore, when we consider colour curve classes, we always assume that these conditions are satisfied.

\subsection{Colour curve contractions}\label{subsec:colour curve contractions}

Suppose that $\calR$ is an extremal ray of $\NE(X)$ which does not contain any wall curve classes. Then $\calR$ is generated by some colour curve class, say $C_{\alpha,\sigma^c}$. Under the contraction $\cont_\calR:X\to X_\calR$, we follow \cite[Section 3.4]{brion1993mori} to describe the coloured fan $\Sigma^c_\calR$ associated to $X_\calR$.

If $u_\alpha\neq 0$, then $N_\calR$ is the coloured lattice $N$ (with the same universal colour set); and if $u_\alpha=0$, then $N_\calR$ has the same underlying lattice as $N$, but the universal colour set is $\calC_\calR=\calC\setminus\{\alpha\}$. The underlying fan $\Sigma_\calR$ is the same as $\Sigma$, and there are two cases for the colour set: if $u_\alpha\neq 0$, then $\calF(\Sigma_\calR^c)=\calF(\Sigma^c)\cup\{\alpha\}$; and if $u_\alpha=0$, then $\calF(\Sigma_\calR^c)=\calF(\Sigma^c)$. 

The contraction $\cont_\calR$ is a Mori fibre space if and only if $u_\alpha=0$. Otherwise, $u_\alpha\neq 0$ and $\cont_\calR$ is birational. In this case, $\cont_\calR$ is divisorial if and only if $u_\alpha$ generates a non-coloured ray of $\Sigma^c$. When $\cont_\calR$ is not divisorial, we can use \cite[Section 4.4]{brion1993mori} to describe the flip. Note that $\Sigma_+^c$ lives on the coloured lattice $N$ since $u_\alpha\neq 0$ in this case. The underlying fan $\Sigma_+$ is obtained from $\Sigma$ by star subdividing through the point $u_\alpha$; see \cite[Section 11.1]{cox2011toric} for this star subdivision. For the colour set, there are two cases: if $u_\alpha$ is not on a ray of $\Sigma$, then $\calF(\Sigma_+^c)=\calF(\Sigma^c)\cup\{\alpha\}$; or if $u_\alpha$ is on a ray $\rho\in\Sigma$, then $\calF(\rho^c)=\{\alpha'\}$ for some $\alpha'\neq\alpha$, and we have $\calF(\Sigma_+^c)=(\calF(\Sigma^c)\cup\{\alpha\})\setminus\{\alpha'\}$.

\subsection{Wall curve contractions}\label{subsec:wall curve contractions}

Now suppose that $\calR$ is an extremal ray of $\NE(X)$ which is generated by a wall curve class $C_\mu$. Under the contraction $\cont_\calR:X\to X_\calR$, we follow \cite[Section 4.6]{brion1993mori} to describe the coloured fan $\Sigma_\calR^c$ associated to $X_\calR$.

If we completely forget about the colour structure, then the fan $\Sigma$ corresponds to a toric variety $X_\Sigma$. The wall $\mu\in\Sigma$ corresponds to an irreducible torus-invariant curve in $X_\Sigma$, and the corresponding curve class generates an extremal ray in $\NE(X_\Sigma)$. Contracting this extremal ray yields a toric variety which corresponds to the fan $\Sigma_\calR$, i.e. the underlying fan of $\Sigma_\calR^c$; see \cite[Section 15.4]{cox2011toric} for a description of the toric contraction and the fan $\Sigma_\calR$ (in their notation, this fan is denoted $\Sigma_0$). 

Since the toric MMP describes what $\Sigma_\calR$ looks like, it also describes the underlying lattice of the coloured lattice $N_\calR$. The universal colour set for $N_\calR$ is $\calC_\calR=\calC\setminus\calC_\Phi$ where $\Phi:N\to N_\calR$ is the coloured lattice map associated to $\cont_\calR:X\to X_\calR$; so $\calC_\Phi$ consists of colours in $\calF(\Sigma^c)$ whose colour points map to zero under $\Phi$. The colour set $\calF(\Sigma_\calR^c)$ is inherited from $\calF(\Sigma^c)$, i.e. it is $\calF(\Sigma^c)\cap\calC_\calR$. 

The contraction $\cont_\calR$ is birational (i.e. not a Mori fibre space) if and only if the corresponding toric contraction of $X_\Sigma$ is birational; see \cite[Proposition 15.4.5]{cox2011toric} for the different types of toric MMP contractions. In this case, $\cont_\calR$ is divisorial if and only if the corresponding toric contraction of $X_\Sigma$ is divisorial. When $\cont_\calR$ is not divisorial, the underlying fan $\Sigma_+$ is the same as the fan corresponding to the flip of the toric variety $X_\Sigma$, and the colour set $\calF(\Sigma_+^c)$ is inherited from $\calF(\Sigma^c)$ so that $\Sigma_+^c$ is simplicial. That is, $\calF(\Sigma_+^c)$ is the same as $\calF(\Sigma^c)$ except we remove all $\alpha\in\calF(\Sigma^c)$ such that $u_\alpha$ is not on a ray of $\Sigma_+$.

\subsection{Flip tower}\label{subsec:flip tower}

Suppose that $\cont_\calR:X\to X_\calR$ is a flipping contraction, and let $\psi:X\dashrightarrow X_+$ denote the flip. Let $\Sigma_*^c$ be the coloured fan on $N$ defined as follows. If $\cont_\calR$ is a colour curve contraction of $C_{\alpha,\sigma^c}$, then let $\Sigma_*^c$ be the same as $\Sigma_+^c$ except that we remove $\alpha$ from the colour set. If $\cont_\calR$ is a wall curve contraction of $C_\mu$, then let $\Sigma_*^c$ be given as follows: the underlying fan is the fan from \cite[Lemma 15.5.7]{cox2011toric}, and the colour set is that of $\Sigma_+^c$. 

In any case, $\Sigma_*$ is the star subdivision of $\Sigma$ and of $\Sigma_+$ through some primitive lattice point $u_*\in N$. If $\cont_\calR$ is a colour curve contraction corresponding to the colour $\alpha$, then $du_*=u_\alpha$ for some $d\in\Z_{>0}$; for notation, let $d:=1$ if $\cont_\calR$ is a wall curve contraction. Note that $(\Cone(u_*),\varnothing)\in\Sigma_*^c$ is a non-coloured ray. It follows that the identity map on $N$ is compatible with $\Sigma_*^c$ and $\Sigma^c$, and with $\Sigma_*^c$ and $\Sigma_+^c$. Therefore, if $X_*$ denotes the horospherical $G/H$-variety corresponding to $\Sigma_*^c$, then we have a commutative diagram of horospherical morphisms
\begin{equation}\label{eq:flip tower}
	\begin{tikzcd}
		& {X_*} \\
		X && {X_+} \\
		& {X_\calR}
		\arrow["{\theta_+}", from=1-2, to=2-3]
		\arrow["\theta"', from=1-2, to=2-1]
		\arrow["{\cont_\calR}"', from=2-1, to=3-2]
		\arrow[from=2-3, to=3-2]
		\arrow["\psi", dashed, from=2-1, to=2-3]
	\end{tikzcd}
\end{equation}
Moreover, $\theta$ and $\theta_+$ are isomorphisms away from the exceptional locus $\Exc(\psi)$. 

We prove the following result which generalizes \cite[Lemma 15.5.7]{cox2011toric} for toric varieties. Let $C\subseteq X$ be a curve which is contracted by $\cont_\calR$. 

\begin{lemma}\label{lemma:flip tower divisor pullbacks}
	Consider the setup above for \eqref{eq:flip tower}. Let $D_*$ denote the $G$-invariant prime divisor in $X_*$ which corresponds to the new non-coloured ray $(\Cone(u_*),\varnothing)\in\Sigma_*^c$. Then, for all $\Q$-divisors $\delta$ on $X$ and $\delta_+:=\psi_*\delta$, we have 
	\begin{align*}
		\theta^*\delta = \theta_+^*\delta_+ - \frac{\delta\cdot C}{d}D_*.
	\end{align*}
	
	\begin{proof}
		Both $X_*$ and $X$ are horospherical $G/H$-varieties, and the map $\theta:X_*\to X$ is a proper birational horospherical morphism which is induced by the identity $N\to N$. Therefore, $\theta^*\delta$ and $\delta$ have the same associated piecewise linear function $\varphi:N_\R\to \R$. Let $\varphi_+:N_\R\to\R$ denote the piecewise linear function for $\delta_+$. 
		
		We may write $\delta=\sum_{(\rho,\varnothing)} a_\rho D_\rho + \sum_{\alpha\in\calC} a_\alpha D_\alpha$ for some $a_\rho,a_\alpha\in\Z$, where the first sum is over all non-coloured rays $(\rho,\varnothing)$ in $\Sigma^c$. Note that the rays of $\Sigma$ and $\Sigma_*$ are the same, except that $\Sigma_*$ has a new ray generated by $u_*$. For each non-coloured ray $(\rho,\varnothing)\in\Sigma^c(1)$, let $D_\rho^*$ denote the $G$-invariant prime divisor in $X_*$ corresponding to $\rho$; and for each colour $\alpha\in\calC$, let $D_\alpha^*$ denote the colour divisor in $X_*$ corresponding to $\alpha$. Then
		\begin{align*}
			\theta^*\delta = \sum_{(\rho,\varnothing)} a_\rho D_\rho^* + \varphi(u_*)D_* + \sum_{\alpha\in\calC} a_\alpha D_\alpha^*
		\end{align*}
		and
		\begin{align*}
			\theta_+^*\delta_+ = \sum_{(\rho,\varnothing)} a_\rho D_\rho^* + \varphi_+(u_*)D_* + \sum_{\alpha\in\calC} a_\alpha D_\alpha^*.
		\end{align*}
		Thus, if suffices to prove that $\varphi_+(u_*)-\varphi(u_*)=\frac{\delta\cdot C}{d}$. 
		
		If it is a wall curve contraction, then we take $d=1$, and the equality $\varphi_+(u_*)-\varphi(u_*)=\delta\cdot C$ follows from the argument in \cite[Lemma 15.5.7]{cox2011toric}. On the other hand, if it is a colour curve contraction corresponding to $\alpha$, then $du_*=u_\alpha$ and $\varphi_+(u_\alpha)=a_\alpha$ because $\alpha\in\calF(\Sigma_+^c)$, so we have
		\begin{align*}
			\varphi_+(u_*)-\varphi(u_*) = \frac{\varphi_+(u_\alpha)-\varphi(u_\alpha)}{d} = \frac{a_\alpha-\varphi(u_\alpha)}{d} = \frac{\delta\cdot C}{d}
		\end{align*}
		where the last equality follows from \eqref{eq:colour curve class}. 
	\end{proof}
\end{lemma}

\section{Preliminary reductions in the proof of Theorem \ref{thm:main theorem}}\label{sec:preliminary reductions in the proof of theorem}

The goal of this paper is to prove the results in Section \ref{sec:introduction}, i.e. we need to prove Proposition \ref{prop:Vojta->canboudned} and Theorem \ref{thm:main theorem}. To start, we quickly prove Proposition \ref{prop:Vojta->canboudned}. 

\begin{proof}[{Proof of Proposition \ref{prop:Vojta->canboudned}}]
	By \cite[Proposition 0.1]{satriano2023erratum-approximating}, we need only show $-K_X$ is big. This is well-known. Choose an ample divisor $A$ on $X$. We have $-K_X=\sum_{D\in\calD(X)} b_D D$ where $b_D\in\Z_{>0}$ (see Remark \ref{rmk:reinterpretation of omega}), and we can write $A=\sum_{D\in\calD(X)} a_D D$ for some $a_D\in\Z$. By choosing $m\in\Z_{>0}$ large enough, we can achieve $\frac{a_D}{m}<b_D$ for each $D\in\calD(X)$. Hence, $-K_X-\frac{1}{m}A$ is effective, so $-K_X$ is big.
\end{proof}

Now we turn our attention to proving Theorem \ref{thm:main theorem}. In this section, we reduce the proof to a simpler case where $p$ is in the exceptional locus of an MMP contraction. This simpler case is stated in Proposition \ref{prop:MMP final step} and proved in Section \ref{sec:proof of proposition MMP final step}.

By \cite[Proposition 4.1]{satriano2021approximating} (which is stated for toric varieties, but holds for horospherical varieties), to prove Theorem \ref{thm:main theorem}, we may assume that the (split) horospherical variety $X$ is projective, $\Q$-factorial, and has at worst terminal singularities. Thus, it suffices to prove the following theorem.

\begin{theorem}\label{thm:main theorem with Q-factorial reduction}
	Let $X$ be a projective terminal $\Q$-factorial horospherical variety, let $p\in X(k)$, and let $D$ be a nef $\Q$-divisor on $X$. If $p$ is canonically bounded, then there exists an irreducible curve $C$ through $p$ which is unibranch at $p$ and
	\begin{align*}
		\alpha_{p,C}(D) \leq \alpha_{p,\{x_i\}}(D)
	\end{align*}
	for all Zariski dense sequences $\{x_i\}\subseteq X(k)$. Moreover, if $X$ is not isomorphic to projective space, then we may choose $C$ so that $-K_X\cdot C\leq \dim(X)$. 
\end{theorem}

Let $X$ be a projective $\Q$-factorial horospherical $G/H$-variety which is canonically bounded at $p\in X(k)$. Let $D$ be a nef $\Q$-divisor on $X$. Let $[C_0],\ldots,[C_\ell]$ be the curve classes which generate the $K_X$-negative extremal rays of $\NE(X)$. Set 
\begin{align}\label{eq:a is minimum of divisor dot curves}
	a := \min_{0\leq i\leq \ell} \frac{D\cdot C_i}{-K_X\cdot C_i}
\end{align}
and, without loss of generality, say $a=\frac{D\cdot C_0}{-K_X\cdot C_0}$. By construction, $D+aK_X$ is nef. By \cite[Lemma 4.3]{satriano2021approximating}, to prove Theorem \ref{thm:main theorem with Q-factorial reduction} for $D$, it suffices to prove the theorem for $D+aK_X$. 

The benefit to working with $D+aK_X$ is that $(D+aK_X)\cdot C_0=0$. Let $\pi:X\to Y$ denote the contraction corresponding to the extremal ray $\R_{\geq 0} C_0\subseteq \NE(X)$. If $\pi$ is a Mori fibre space or a divisorial contraction, then there exists a nef $\Q$-divisor $D'$ on $Y$ such that $D+aK_X=\pi^*D'$. On the other hand, if $\pi$ is a flipping contraction, then let $\psi:X\dashrightarrow X_+$ denote the associated flip. If we let $D':=\psi_*(D+aK_X)$, then Lemma \ref{lemma:flip tower divisor pullbacks} tells us that $\theta^*(D+aK_X)=\theta_+^*D'$, where we are using the notation of \eqref{eq:flip tower}. Since $\theta$ and $\theta_+$ are proper and surjective, $D+aK_X$ being nef implies that $\theta^*(D+aK_X)$ is nef, which implies that $D'$ is nef. 

To simplify the notation, we let $\psi:X\dashrightarrow X'$ denote the elementary MMP step corresponding to the extremal ray $\R_{\geq 0} C_0$, i.e. if $\pi$ is a Mori fibre space or a divisorial contraction, then we let $X':=Y$ and $\psi:=\pi$; on the other hand, if $\pi$ is a flipping contraction, then we let $\psi$ denote the associated flip and $X':=X_+$. In any case, we have shown that there is a nef $\Q$-divisor $D'$ on $X'$ for which $D+aK_X=\psi^*D'$. 

In Proposition \ref{prop:MMP reduction-special-case} below, we show that the proof of Theorem \ref{thm:main theorem with Q-factorial reduction} reduces to the case where $p\in\Exc(\psi)$. Before proving this reduction, we require the following two preliminary results. 

\begin{proposition}\label{prop:MMP preserves canonically bounded}
	Let $X$ be a projective terminal $\Q$-factorial horospherical variety, and let $\psi\colon X\dashrightarrow X'$ be a birational elementary MMP step. If $p\in X(k)\setminus \Exc(\psi)$ is a canonically bounded point of $X$, then $\psi(p)$ is a canonically bounded point of $X'$.
	
	\begin{proof}
		This follows from an almost identical proof of \cite[Proposition 4.8]{satriano2021approximating}. Note that this proof relies on \cite[(4.5),(4.6)]{satriano2021approximating}, which are generalized to horospherical varieties in \eqref{eq:flip tower} and Lemma \ref{lemma:flip tower divisor pullbacks}. 
	\end{proof} 
\end{proposition}

\begin{proposition}\label{prop:MMP reduction}
	Let $X$ be a projective terminal $\Q$-factorial horospherical variety, and let $\psi\colon X\dashrightarrow X'$ be a birational elementary MMP step corresponding to the extremal ray $\calR$. Let $D$ and $D'$ be nef $\Q$-divisors on $X$ and $X'$, respectively, such that $D\in \calR^\perp$ and $D=\psi^*D'$. If $p\in X(k)\setminus\Exc(\psi)$ is a canonically bounded point and Theorem \ref{thm:main theorem with Q-factorial reduction} holds for $(X',\psi(p),D')$, then it holds for $(X,p,D)$. 
	
	\begin{proof}
The proof is nearly the same as \cite[Proposition 4.9]{satriano2021approximating}, which is handled in \cite[Section 5]{satriano2021approximating}. Note that \cite[Lemma 5.1]{satriano2021approximating} is a statement about projective space; \cite[Lemma 5.2]{satriano2021approximating} holds for horospherical varieties; and \cite[Lemma 5.3]{satriano2021approximating} holds for horospherical varieties by \eqref{eq:flip tower} and Lemma \ref{lemma:flip tower divisor pullbacks}. 
	\end{proof}
\end{proposition}

\begin{proposition}\label{prop:MMP reduction-special-case}
	Suppose that Theorem \ref{thm:main theorem with Q-factorial reduction} holds for all $(X,p,D)$ where $X$ is a projective terminal $\Q$-factorial horospherical variety, $\psi\colon X\dashrightarrow X'$ is an elementary MMP step corresponding to the extremal ray $\calR$, $D$ is a nef $\Q$-divisor on $X$ such that $D\in\calR^\perp$, and $p\in X(k)\cap \Exc(\psi)$ is a canonically bounded point. Then Theorem \ref{thm:main theorem with Q-factorial reduction} holds for all triples.
	
	\begin{proof}
		Start with a projective terminal $\Q$-factorial horospherical variety $X_1$, let $D_1$ be a nef $\Q$-divisor on $X_1$, and $p_1\in X_1(k)$ a canonically bounded point. Let $a_1$ be as in \eqref{eq:a is minimum of divisor dot curves}, so that $D_1+a_1K_{X_1}\in \calR_1^\perp$ for some $K_{X_1}$-negative extremal ray $\calR_1$ of $\NE(X_1)$. Now let $\psi\colon X_1\dashrightarrow X_2$ be the elementary MMP step associated to $\calR_1$, and let $D_2$ be the nef $\Q$-divisor on $X_2$ such that $D_1+a_1K_{X_1}=\psi_1^*D_2$. Repeating this procedure, we arrive at the following data: we have a sequence of elementary MMP steps
		\begin{equation*}
			\begin{tikzcd}
				{X_1} & {X_2} & \cdots & {X_{m+1}}
				\arrow["{\psi_1}", dashed, from=1-1, to=1-2]
				\arrow["{\psi_2}", dashed, from=1-2, to=1-3]
				\arrow["{\psi_m}", dashed, from=1-3, to=1-4]
			\end{tikzcd}
		\end{equation*}
		and a sequence of canonically bounded points $p_i\in X_i(k)$ such that $p_m\in\Exc(\psi_m)$, and $p_i\notin\Exc(\psi_i)$ and $p_{i+1}=\psi_i(p_i)$ for all $1\leq i<m$. Furthermore, for each $1\leq i\leq m$, we have a nef $\Q$-divisor $D_i$ on $X_i$, and we have $a_i\in\Q_{\geq 0}$ such that $D_i+a_iK_{X_i}=\psi_i^*D_{i+1}$ is a nef $\Q$-divisor perpendicular to the $K_{X_i}$-negative extremal ray corresponding to $\psi_i$. 
		
		By Proposition \ref{prop:MMP reduction}, the case of $(X_i,p_i,D_i+a_iK_{X_i})$ follows from that of $(X_{i+1},p_{i+1},D_{i+1})$. Therefore, to prove Theorem \ref{thm:main theorem with Q-factorial reduction} for the triple $(X_1,p_1,D_1)$, it suffices to prove it for the triple $(X_m,p_m,D_m)$ where we have $p_m\in X(k)\cap\Exc(\psi_m)$. 
	\end{proof}
\end{proposition}

In conclusion, we have reduced the proof of Theorem \ref{thm:main theorem with Q-factorial reduction} to proving the following.

\begin{proposition}\label{prop:MMP final step}
	Let $X$ be a projective terminal $\Q$-factorial horospherical variety, and let $\psi\colon X\dashrightarrow X'$ be an elementary MMP step corresponding to the extremal ray $\calR$. If $D$ is a nef $\Q$-divisor on $X$ such that $D\in\calR^\perp$ and $p\in X(k)\cap \Exc(\psi)$ is a canonically bounded point, then Theorem \ref{thm:main theorem with Q-factorial reduction} holds for $(X,p,D)$. 
\end{proposition}

\section{An inequality on positive roots}\label{sec:an inequality on positive roots}

This section establishes an inequality that is used throughout the proof of Proposition \ref{prop:MMP final step}. Recall that the parabolic subgroup $P\subseteq G$ corresponds uniquely to a subset $I\subseteq S$, where $S$ is the set of simple roots for $G$ (see \cite[Section 2.4]{monahan2023overview} for a review of this correspondence), and recall that $\calC:=S\setminus I$. Let $R^+$ denote the set of positive roots in $G$ (relative to $B$ and $T$), and let $R_I$ denote the set of roots generated by $I$. In this section, we use $\langle\cdot,\cdot\rangle$ to denote the dual pairing between roots and dual roots (see \cite[Page 6]{pasquier2009introduction} for a quick reference on this dual pairing). We prove the following result. 

\begin{proposition}\label{prop:root inequality}
	Using the notation immediately above, we have the following inequality:
	\begin{align}\label{eq:root inequality}
		\sum_{\alpha\in\calC}\sum_{\beta\in R^+\setminus R_I} \langle \beta,\alpha^\vee\rangle \leq \#\calC + \#(R^+\setminus R_I).
	\end{align}
	Moreover, this is equality if and only if $G/P$ is a product of projective spaces.
\end{proposition}

Throughout this section, we say that $\alpha\in S$ is in the \define{support} of $\beta\in R^+$ if $\alpha$ appears with nonzero coefficient when $\beta$ is expressed as a nonnegative integer combination of simple roots. Since the dual pairing of non-adjacent roots is trivial, to prove Proposition \ref{prop:root inequality} we may assume going forward that $G$ is semisimple and has a connected Dynkin diagram. Note that Proposition \ref{prop:root inequality} holds when $I=S$ since both sides are $0$ and $G/P$ is a point (which is an empty product of projective spaces), so we may assume $I\subsetneq S$. 

We must show that
\begin{align*}
	\omega(I) := \#\calC + \#(R^+\setminus R_I) - \sum_{\alpha\in\calC}\sum_{\beta\in R^+\setminus R_I} \langle \beta,\alpha^\vee\rangle
\end{align*}
is nonnegative. We start by proving that this holds for maximal subsets $I\subsetneq S$. 

\begin{lemma}\label{lemma:root inequality for maximal}
	Suppose that $I\subsetneq S$ is maximal. Then $\omega(I)\geq 0$. Moreover, we have $\omega(I)=0$ if and only if $G/P$ is a projective space. 
	
	\begin{proof}
		The following are the possible Dynkin diagram types for $G$: $\mathbf{A}_\ell$ for $\ell\geq 1$, $\mathbf{B}_\ell$ for $\ell\geq 3$, $\mathbf{C}_\ell$ for $\ell\geq 2$, $\mathbf{D}_\ell$ for $\ell\geq 4$, $\mathbf{E}_6$, $\mathbf{E}_7$, $\mathbf{E}_8$, $\mathbf{F}_4$, and $\mathbf{G}_2$. We prove the result in each case. 
		
		For each of the special cases $\mathbf{E}_6$, $\mathbf{E}_7$, $\mathbf{E}_8$, $\mathbf{F}_4$, and $\mathbf{G}_2$, one can check that $\omega(I)>0$ using a computer (the first-named author has posted code for GAP on GitHub \cite{monahan2023github-root}).
		
		It remains to consider types $\mathbf{A}_\ell$-$\mathbf{D}_\ell$. Since $I\subsetneq S$ is maximal, we have $\calC=\{\alpha\}$ for some $\alpha\in S$, and $R^+\setminus R_I$ consists of all positive roots $\beta$ with $\alpha$ in its support. So we must show 
		\begin{align}\label{eq:root inequality maximal case}
			\sum_{\beta\in R^+\setminus R_I} \langle \beta,\alpha^\vee\rangle \leq \#(R^+\setminus R_I)+1.
		\end{align}
		
		First, note that every $\beta\in R^+\setminus R_I$ can be written as $\beta=\sum_{\gamma\in S} c_{\beta,\gamma} \gamma$ for some $c_{\beta,\gamma}\in\{0,1,2\}$ with $c_{\beta,\alpha}\in\{1,2\}$. Note that $\langle \alpha,\alpha^\vee\rangle=2$, and for all other $\gamma\in S$ we have $\langle \gamma,\alpha^\vee\rangle\leq 0$. 
		
		To prove \eqref{eq:root inequality maximal case}, it suffices to prove the following two things: (i) for all $\beta\in R^+\setminus R_I$ with $\beta\neq\alpha$ and $c_{\beta,\alpha}=1$, we have $\langle \beta,\alpha\rangle\leq 1$; and (ii) for all $\beta\in R^+\setminus R_I$ with $c_{\beta,\alpha}=2$, we have either $\langle \beta,\alpha^\vee\rangle\leq 1$, or $\langle \beta,\alpha^\vee\rangle\leq 2$ and there exists $\beta'\in R^+\setminus R_I$ such that $\langle \beta',\alpha^\vee\rangle\leq 0$. Indeed, if (i) and (ii) hold, then the average value of $\langle \beta,\alpha^\vee\rangle$ over all $\beta\in R^+\setminus R_I$ with $\beta\neq \alpha$ is bounded above by $1$, and thus 
		\begin{align*}
			\frac{1}{\#((R^+\setminus R_I)\setminus\{\alpha\})}\sum_{\beta\in (R^+\setminus R_I)\setminus\{\alpha\}} \langle \beta,\alpha^\vee\rangle &\leq 1 
			\\ \sum_{\beta\in (R^+\setminus R_I)\setminus\{\alpha\}} \langle \beta,\alpha^\vee\rangle &\leq \#(R^+\setminus R_I)-1
			\\ \sum_{\beta\in R^+\setminus R_I} \langle\beta,\alpha^\vee\rangle &\leq \#(R^+\setminus R_I)+1
		\end{align*}
		where in the last step we added $\langle \alpha,\alpha^\vee\rangle=2$ to both sides. 
		
		Now we prove (i). In this case, there is some $\gamma_0\in S$ adjacent to $\alpha$ in the Dynkin diagram of $G$ which satisfies $c_{\beta,\gamma_0}=1$. Hence 
		\begin{align*}
			\langle \beta,\alpha^\vee\rangle \leq c_{\beta,\alpha}\langle \alpha,\alpha^\vee\rangle + c_{\beta,\gamma_0}\langle \gamma_0,\alpha^\vee\rangle = (1)(2)+(1)(-1)=1.
		\end{align*}
		
		Now we prove (ii). Note that $c_{\beta,\alpha}=2$ cannot occur in type $\mathbf{A}_\ell$. So it remains to check types $\mathbf{B}_\ell$-$\mathbf{D}_\ell$.
		
		In type $\mathbf{D}_\ell$, if $c_{\beta,\alpha}=2$, then there must be (two or three) simple roots $\gamma$ adjacent to $\alpha$ whose coefficients $c_{\beta,\gamma}$ sum to at least $3$; in each case we have $\langle \beta,\gamma\rangle=-1$. Hence
		\begin{align*}
			\langle \beta,\alpha^\vee\rangle \leq c_{\beta,\alpha}\langle \alpha,\alpha^\vee\rangle + (3)(-1) = (2)(2)+(3)(-1)=1.
		\end{align*}
		
		In type $\mathbf{B}_\ell$ (for $\ell\geq 3$), if $c_{\beta,\alpha}=2$, then $\alpha$ cannot be the longest root, and we have the following two cases. First, if $\alpha$ is not the shortest root, then the two simple roots $\gamma_0,\gamma_1$ adjacent to $\alpha$ must satisfy $c_{\beta,\gamma_0}+c_{\beta,\gamma_1}\geq 3$; note that $\langle \gamma_i,\alpha^\vee\rangle=-1$ for $i=1,2$. In this case,
		\begin{align*}
			\langle \beta,\alpha^\vee\rangle \leq c_{\beta,\alpha}\langle \alpha,\alpha^\vee\rangle + c_{\beta,\gamma_0}\langle \gamma_0,\alpha^\vee\rangle + c_{\beta,\gamma_1}\langle \gamma_1,\alpha^\vee\rangle \leq (2)(2)+(3)(-1) = 1
		\end{align*}
		On the other hand, suppose that $\alpha$ is the shortest root, so there is a unique simple root $\gamma_0$ adjacent to $\alpha$; note that $\langle \gamma_0,\alpha^\vee\rangle=-2$. In this case, 
		\begin{align*}
			\langle \beta,\alpha^\vee\rangle \leq c_{\beta,\alpha}\langle \alpha,\alpha^\vee\rangle + c_{\beta,\gamma_0}\langle \gamma_0,\alpha^\vee\rangle = (2)(2) + c_{\beta,\gamma_0}(-2) = 2(2-c_{\beta,\gamma_0}).
		\end{align*}
		If $c_{\beta,\gamma_0}=1$, then the above implies $\langle \beta,\alpha^\vee\rangle \leq 2$, so we may take $\beta':=\beta-\gamma_0\in R^+\setminus R_I$ to obtain $\langle \beta',\alpha\rangle\leq 0$ since $c_{\beta',\gamma_0}=2$. On the other hand, if $c_{\beta,\gamma_0}=2$, then we automatically get $\langle \beta,\alpha^\vee\rangle\leq 0$. 
		
		Finally, in type $\mathbf{C}_\ell$, if $c_{\beta,\alpha}=2$, then $\alpha$ cannot be the longest root, and we have the following two cases. First, if $\alpha$ is not the shortest root, then we get $\langle \beta,\alpha^\vee\rangle\leq 1$ as in the $\mathbf{B}_\ell$ case. On the other hand, suppose that $\alpha$ is the shortest root. Then the simple root $\gamma_0$ adjacent to $\alpha$ must satisfy $c_{\beta,\gamma_0}=2$; note that $\langle \gamma_0,\alpha^\vee\rangle=-1$. Hence
		\begin{align*}
			\langle \beta,\alpha^\vee\rangle \leq c_{\beta,\alpha}\langle \alpha,\alpha^\vee\rangle + c_{\beta,\gamma_0}\langle \gamma_0,\alpha^\vee\rangle = (2)(2) + (2)(-1) = 2.
		\end{align*}
		In this case, we may take $\beta':=\beta-\alpha\in R^+\setminus R_I$, and a similar computation to the one above yields $\langle \beta',\alpha^\vee\rangle\leq 0$ since $c_{\beta',\alpha}=1$. 
		
		Lastly, it remains to prove the ``moreover" part in the lemma statement. It is easy to verify that $\omega(I)=0$ if and only if one of the following cases holds: the Dynkin diagram of $G$ is type $\mathbf{A}_\ell$ and $\alpha$ is an end simple root in the diagram, or type $\mathbf{C}_\ell$ and $\alpha$ is the longest simple root. These two cases are precisely the cases when the flag variety $G/P$ is a projective space.
	\end{proof}
\end{lemma}

Our next goal is to show that $\omega(I)$ is minimized when $I$ is maximal by proving Lemma \ref{lemma:monotonicity of root inequality}. We do so after a preliminary result.

\begin{remark}\label{rmk:reinterpretation of omega}
	By \cite[Theorem 6.1.2]{perrin2018sanya}, we may write
	\begin{align*}
		-K_{G/P}=\sum_{\alpha\in\calC} b_\alpha D_\alpha
	\end{align*}
	where $D_\alpha\subseteq X$ is the colour divisor corresponding to $\alpha$ and 
	\begin{align}\label{eq:coefficients of G/P canonical}
		b_\alpha=\sum_{\beta\in R^+\setminus R_I} \langle \beta,\alpha^\vee\rangle\in\Z_{>0}.
	\end{align}
	Note that
	\begin{align*}
		\omega(I) = \#\calC + \dim(G/P) - \sum_{\alpha\in\calC} b_\alpha.
	\end{align*}
\end{remark}

\begin{lemma}\label{lemma:inequality for coefficients of flag canonical}
	Let $I'\subsetneq I$, $P':=P_{I'}$ be the parabolic corresponding to $I'$, and $\calC':=S\setminus I'$. Let $b_\alpha$ (resp. $b'_\alpha$) be as in \eqref{eq:coefficients of G/P canonical} corresponding to $I$ (resp. $I'$). Then, for each $\alpha\in\calC$, we have $b_\alpha'\leq b_\alpha$. Furthermore, we have
	\begin{align*}
		\sum_{\alpha\in\calC} b_\alpha' < \sum_{\alpha\in\calC} b_\alpha.
	\end{align*}
	
	\begin{proof}
		Fix $\alpha\in \calC=S\setminus I$. By \eqref{eq:coefficients of G/P canonical}, we have
		\begin{align*}
			b_\alpha' = b_\alpha + \sum_{\beta\in R_I^+\setminus R_{I'}^+} \langle \beta,\alpha^\vee\rangle. 
		\end{align*}
		Since $\alpha\notin I$, $\alpha$ is not in the support of any $\beta\in R_I^+\setminus R_{I'}^+$. Hence, $\langle \beta,\alpha^\vee\rangle\leq 0$ for each $\beta\in R_I^+\setminus R_{I'}^+$, so the above equation implies $b_\alpha'\leq b_\alpha$. 
		
		Lastly, since we are assuming that the Dynkin diagram of $G$ is connected, and since $I\setminus I'\neq\varnothing$, we may choose $\beta\in R_I^+\setminus R_{I'}^+$ and $\gamma\in\calC$ so that $\langle \beta,\gamma^\vee\rangle<0$. Therefore, the above equation implies $b_\gamma'<b_\gamma$, and thus $\sum_{\alpha\in\calC} b_\alpha' < \sum_{\alpha\in\calC} b_\alpha$.
	\end{proof}
\end{lemma}

\begin{lemma}\label{lemma:monotonicity of root inequality}
	If $I'\subsetneq I\subsetneq S$, then $\omega(I)<\omega(I')$. 
	
	\begin{proof}
		It suffices to consider the special case where $\#(I\setminus I')=1$, say $I\setminus I'=\{\gamma\}$. Indeed, the general case follows from inductively applying this special case.
		
		Consider the notation in the set up of Lemma \ref{lemma:inequality for coefficients of flag canonical}. Let $G':=P/R_u(P)$ (where $R_u(P)$ is the unipotent radical of $P$), so $G'$ is connected and reductive, the set of simple roots for $G$ is $I$, and the set of positive roots is $R_I^+$. Then $P/P'$ is isomorphic to the flag variety $G'/P''$ where $P''$ is the parabolic subgroup of $G'$ corresponding to the subset $I'\subseteq I$ of simple roots. Since $I'\subsetneq I$ is a maximal subset, Lemma \ref{lemma:root inequality for maximal} implies that $\omega(G'/P'')\geq 0$. 
		
		Note that $\calC'\setminus\calC=I\setminus I'=\{\gamma\}$ is the universal colour set of $G'/P''$. As in Remark \ref{rmk:reinterpretation of omega}, we may write $-K_{G'/P''}=b_\gamma''D_\gamma''$ for some $b_\gamma''\in\Z_{>0}$, where $D_\gamma''$ is the unique colour divisor of $G'/P''$. We first show that $b_\gamma'=b_\gamma''$. Indeed, by \eqref{eq:coefficients of G/P canonical} we have
		\begin{align*}
			b_\gamma'-b_\gamma'' = \sum_{\beta\in R^+\setminus R_{I'}} \langle \beta,\gamma^\vee\rangle - \sum_{\beta\in R_I^+\setminus R_{I'}} \langle \beta,\gamma^\vee\rangle = \sum_{\beta\in R^+} \langle \beta,\gamma^\vee\rangle - \sum_{\beta\in R_I^+} \langle \beta,\gamma^\vee\rangle = 2-2 = 0
		\end{align*}
		since $\sum_{\beta\in R^+} \langle \beta,\gamma^\vee\rangle$ and $\sum_{\beta\in R_I^+} \langle \beta,\gamma^\vee\rangle$ are the coefficients of the colour divisors corresponding to $\gamma$ in the anticanonical divisors of $G/B$ and $G'/B'$ ($B'$ denoting the corresponding Borel subgroup in $G'$), respectively, which are known to be exactly $2$ (see \cite[Proposition 2.2.8(iv)]{brion2004lectures}). 
		
		Using Remark \ref{rmk:reinterpretation of omega} and Lemma \ref{lemma:inequality for coefficients of flag canonical}, we have
		\begin{align*}
			\omega(I')-\omega(I) =& \dim(G/P')-\dim(G/P) + \#\calC'-\#\calC + \sum_{\alpha\in\calC'} (-b_\alpha')-\sum_{\alpha\in\calC} (-b_\alpha)
			\\=& \dim(G'/P'') + \#(\calC'\setminus \calC) + (-b_\gamma') + \sum_{\alpha\in\calC} (b_\alpha-b_\alpha')
			\\=& \dim(G'/P'') + \#(\calC'\setminus \calC) + (-b_\gamma'') + \sum_{\alpha\in\calC} (b_\alpha-b_\alpha')
			\\=& \omega(G'/P'')+\sum_{\alpha\in\calC} (b_\alpha-b_\alpha') > 0. \qedhere
		\end{align*}
	\end{proof}
\end{lemma}

\begin{proof}[Proof of Proposition \ref{prop:root inequality}]
	The inequality \eqref{eq:root inequality} follows by combining Lemma \ref{lemma:root inequality for maximal} and Lemma \ref{lemma:monotonicity of root inequality}, and these results imply that we have equality if and only if $G/P$ is a product of projective spaces.  
\end{proof}

\section{Results on Picard number $1$ horospherical varieties}\label{sec:results on Picard number 1 horospherical varieties}

This section is dedicated to proving Proposition \ref{prop:Picard 1 - full}, stated below, on projective $\Q$-factorial horospherical varieties with Picard number $1$, which is the main input in the proof of Proposition \ref{prop:MMP final step}. 

We begin by classifying these horospherical varieties with Picard number $1$.

\begin{lemma}\label{lemma:Picard number 1 horospherical varieties}
	Suppose that $X$ is a projective $\Q$-factorial horospherical $G/H$-variety. Then $X$ has Picard number $1$ if and only if one of the following holds:
	\begin{enumerate}
		\item $\rank(N)=0$, $\#\calC=1$, and $\Sigma^c=0^c$. Said another way, $X$ is a flag variety with exactly one colour divisor. 
		\item $\rank(N)>0$, $\calF(\Sigma^c)=\calC$, and $\#\Sigma^c(1)=\rank(N)+1$.
	\end{enumerate}
	
	\begin{proof}
		Since $X$ is $\Q$-factorial, the Picard number $\rank(\Pic(X))$ is equal to $\rank(\Cl(X))$, which is equal to $\#\calD(X)-\rank(N)$. Since $\Sigma^c$ is simplicial, we have an injection $\calF(\Sigma^c)\hookrightarrow \Sigma^c(1)$ which sends $\alpha\mapsto (\Cone(u_\alpha),\{\alpha\})$. It follows that $\#\calD(X)=\#\Sigma^c(1)+\#(\calC\setminus\calF(\Sigma^c))$. Since $X$ is projective, $\Sigma^c$ is complete, and thus $\#\Sigma^c(1)\geq \rank(N)+1$ or $\rank(N)=0$ (in which case $\#\Sigma^c(1)=0$ and $\calF(\Sigma^c)=\varnothing$). Hence
		\begin{align*}
			\rank(\Pic(X))=1 &\iff (\#\Sigma^c(1)-\rank(N))+\#(\calC\setminus\calF(\Sigma^c))=1 
		\end{align*}
		and the final equation is satisfied if and only if conditions (1) or (2) are met. 
	\end{proof}
\end{lemma}

The following is the main result of this section. It yields a strong version of Theorem \ref{thm:main theorem with Q-factorial reduction} for these horospherical varieties with Picard number $1$. 

\begin{proposition}\label{prop:Picard 1 - full}
	Let $W$ be a projective $\Q$-factorial horospherical variety with Picard number $1$, and let $p\in W(k)$. Then there exists a unibranch rational curve $C\subseteq W$ through $p$ which satisfies $-K_W\cdot C\leq 1+\dim(W)$. Furthermore, if $W$ has terminal singularities and is not isomorphic to projective space, then $C$ can be chosen to additionally satisfy $-K_W\cdot C\leq \dim(W)$. 
\end{proposition}

From Lemma \ref{lemma:Picard number 1 horospherical varieties}, we know that there are two types of projective $\Q$-factorial horospherical varieties with Picard number $1$. The first is flag varieties with a single colour divisor. In this case, Proposition \ref{prop:Picard 1 - full} becomes the following.

\begin{proposition}\label{prop:Picard 1 - flag variety result}
	Proposition \ref{prop:Picard 1 - full} holds when $W$ is a flag variety with Picard number $1$. 
	
	\begin{proof}
		Since $W=G/P$ is a single $G$-orbit, we may translate our curve by a $G(k)$-point to ensure that it passes through any particular point in $W$. Thus, we may assume that $p$ is the base point $eP\in W=G/P$. By assumption, $W$ has exactly one colour divisor $D_\alpha$ with $\calC=\{\alpha\}$. We may choose the dual Schubert curve $C\subseteq W$ so that $D_\alpha\cdot C=1$. This curve $C$ is unibranch, rational, and passes through $p$. Furthermore, Remark \ref{rmk:reinterpretation of omega} implies that $-K_W=aD_\alpha$ where $a$ equals the left-hand side of \eqref{eq:root inequality}. Therefore, Proposition \ref{prop:root inequality} and the fact that the right-hand side of \eqref{eq:root inequality} is equal to $\dim(W)+1$ imply that 
		\begin{align*}
			-K_W\cdot C=aD_\alpha\cdot C=a\leq \dim(W)+1.
		\end{align*}
		If $W$ is not isomorphic to projective space, then the above inequality is a strict inequality of integers, so we can strengthen it to $-K_W\cdot C\leq\dim(W)$. 		
	\end{proof}
\end{proposition}

Now we turn to Proposition \ref{prop:Picard 1 - full} for varieties of the second type in Lemma \ref{lemma:Picard number 1 horospherical varieties}. 
%We give these varieties a special name: if $W$ is a projective $\Q$-factorial horospherical $G/H$-variety with Picard number $1$, and if $N=N(G/H)$ has positive rank, then we call $W$ a \define{P1 horospherical variety} \sean{rename these? just call them ``varieties of type Lemma \ref{lemma:Picard number 1 horospherical varieties}(2)" or something?}. 
%The following result handles Proposition \ref{prop:Picard 1 - full} for such varieties. 
Together with Proposition \ref{prop:Picard 1 - flag variety result}, this implies Proposition \ref{prop:Picard 1 - full}. 

\begin{proposition}\label{prop:Picard 1 - P1 result}
	Let $W$ be as in Lemma \ref{lemma:Picard number 1 horospherical varieties}(2), and let $p\in W(k)$. Then there is a unibranch rational curve $C\subseteq W$ through $p$ which satisfies the following properties:
	\begin{enumerate}
		\item There exists a $G$-orbit closure $Z\subseteq W$ and a curve $C_0\subseteq Z$ which is the closure of the orbit of a $Z(k)$-point under the action of a one-parameter subgroup of $\Aut^G(Z)$ such that $C$ is the translate of $C_0$ by a $G(k)$-point.
		\item $-K_W\cdot C \leq 1+\dim(W)$.
		\item If $W$ has terminal singularities and is not isomorphic to projective space, then $C$ can be chosen to additionally satisfy $-K_W\cdot C\leq \dim(W)$. 
	\end{enumerate}
\end{proposition}

In the case where $W$ is a toric variety, Proposition \ref{prop:Picard 1 - P1 result} follows from \cite[Proposition 6.1]{satriano2021approximating}. Therefore, to complete the proof, we may choose to assume that $W$ is not toric. The rest of this section is dedicated to handling this case.

\subsection{Divisor-curve pairing on orbit closures}\label{subsec:divisor-curve pairing on orbit closures}

We begin with a preliminary result about divisor-curve intersection pairing on orbit closures of horospherical varieties; compare with \cite[Lemma 12.5.2]{cox2011toric} for toric varieties. 

Suppose that $X$ is a projective $\Q$-factorial horospherical $G/H$-variety, and suppose that $n:=\rank(N)>0$. Let $v_0,\ldots,v_l$ denote the primitive generators for the rays in $\Sigma(X)$. For each $0\leq i\leq l$, if the coloured ray in $\Sigma^c(X)$ generated by $v_i$ has colour set $\{\alpha\}$, then set $w_i:=u_\alpha$; otherwise set $w_i:=v_i$. Note that $w_i$ is a positive integer multiple of $v_i$, which we denote by $c_i\in \Z_{>0}$. Each $w_i$ corresponds to a $B^-$-invariant prime divisor $D_i\subseteq X$ in the following way: if $w_i$ generates a non-coloured ray in $\Sigma^c(X)$, then $D_i$ denotes the corresponding $G$-invariant prime divisor; and if $w_i$ generates a coloured ray whose colour set consists of a single colour $\alpha$, then $D_i:=D_\alpha^X$. In total, we have
\begin{align*}
	\calD(X) = \{D_0,\ldots,D_l\}\cup\{D_\alpha^X:\alpha\in\calC\setminus\calF(X)\}.
\end{align*}

Now consider a $G$-orbit $\calO\subseteq X$ corresponding to some $\tau^c\in\Sigma^c(X)$, and let $Z:=\ol{\calO}$; we assume that $\tau^c\neq 0^c$, i.e. $Z\neq X$. Let $N'$ denote the coloured lattice on which $\Sigma^c(Z)$ lives, and let $\phi:N\to N'$ denote the quotient map (see \cite[Section 5.4]{monahan2023overview}). Note that the coloured rays of $\Sigma^c(Z)$ correspond to the coloured rays of $\Sigma^c(X)$ which are in $\Star(\tau^c)$ but not in $\tau^c$; note that $\Star(\tau^c)$ denotes the sub-coloured fan of $\Sigma^c(X)$ whose underlying fan is $\Star(\tau)$ (see \cite[(3.2.8)]{cox2011toric}). Without loss of generality, the rays of $\Star(\tau)$ are generated by $v_0,\ldots,v_b$, and the rays of $\tau$ are generated by $v_0,\ldots,v_a$ (so $0\leq a\leq b\leq l$). Then the rays of $\Sigma(Z)$ are generated by the images of $v_{a+1},\ldots,v_b$ under $\phi$. Similar to the previous paragraph, let $D_i'\subseteq Z$ denote the $B^-$-invariant prime divisor corresponding to the $i$-th coloured ray in $\Sigma^c(Z)$. Then
\begin{align*}
	\calD(Z) = \{D_{a+1}',\ldots,D_b'\}\cup\{D_\alpha^Z:\alpha\in\calC\setminus\calF(\Star(\tau^c))\}.
\end{align*}

\begin{lemma}\label{lemma:divisor-curve pairing on orbit closure}
	Consider the notation immediately above. Choose $a+1\leq i\leq b$ and $\alpha\in\calC\setminus\calF(\Star(\sigma^c))$. If $C$ is a curve contained in $Z$, then we have
	\begin{align*}
		D_i\cdot C\leq D_i'\cdot C \quad\text{and}\quad D_\alpha^X\cdot C \leq D_\alpha^Z\cdot C.
	\end{align*}
	
	\begin{proof}
		Let $\iota:Z\to X$ be the inclusion map. For a $B^-$-invariant $\Q$-Cartier divisor $D$ on $X$ (resp. $D'$ on $Z$), let $\varphi_D:N_\R\to\R$ (resp. $\varphi_{D'}:N'_\R\to \R$) denote the associated piecewise linear function.
		
		First consider $D:=D_i$; let $D':=D_i'$. Then $\iota^*D$ has the associated piecewise linear function $\varphi_{\iota^*D}:N'_\R\to\R$ which sends $\phi(w_j)\mapsto \varphi_D(w_j)$ for each $a+1\leq j\leq b$. Note that $\phi(w_j)=d_jw_j'$ for some $d_j\in\Z_{>0}$, where $w_j'\in N'$ is the analogue of $w_j$ for $Z$. Thus
		\begin{align*}
			\varphi_{\iota^*D}(w_j') = \frac{1}{d_j}\varphi_D(w_j) = \frac{1}{d_j}\varphi_{D'}(w_j')
		\end{align*}
		for each $a+1\leq j\leq b$, so $\iota^*D$ and $\frac{1}{d_i}D'$ have the same piecewise linear function. Hence 
		\begin{align*}
			D\cdot C = \iota^*D\cdot C = \frac{1}{d_i}D'\cdot C \leq D'\cdot C.
		\end{align*}
		
		To finish the proof, we consider the case where $D:=D_\alpha^X$; let $D':=D_\alpha^Z$. In this case, $\iota^*D$ and $D'$ have the same piecewise linear function (the trivial function), so
		\begin{align*}
			D\cdot C &= \iota^*D\cdot C = D'\cdot C. 
			\qedhere
		\end{align*}
	\end{proof}
\end{lemma}

\subsection{Proving Proposition \ref{prop:Picard 1 - P1 result}}

The first main step in the proof is the following essential inequality. 

\begin{lemma}\label{lemma:divisor-curve bound}
	Let $W$ be a horospherical $G/H$-variety as in Lemma \ref{lemma:Picard number 1 horospherical varieties}(2), and let $p\in W(k)$. Then there exists a curve $C\subseteq W$ through $p$ which satisfies Proposition \ref{prop:Picard 1 - P1 result}(1), and satisfies $D\cdot C\leq 1$ for all $D\in\calD(W)$. 
	
	\begin{proof}		
		We prove the result via induction on $n:=\rank(N)>0$. We use the notation of Section \ref{subsec:divisor-curve pairing on orbit closures} for the vectors $v_i,w_i\in N$ and the corresponding $B^-$-invariant prime divisors $D_i\subseteq W$; note that $l=n$ and $\calD(W)=\{D_0,\ldots,D_n\}$. There exists relatively prime positive integers $a_0,\ldots,a_n$ such that $\sum_i a_iw_i=0$. Without loss of generality, say $a_0=\max_i\{a_i\}$.
		
		Let $\lambda:\Gm\to \Aut^G(W)$ be the one-parameter subgroup corresponding to $v_0\in N$, and let $C_\lambda:=\ol{\{\lambda\cdot eH\}}\subseteq W$ be the closure of the $\lambda$-orbit of the base point $eH\in W$. Note that $\{\lambda\cdot eH\}\cong\A^1\setminus\{0\}$ is contained in the open $B^-$-orbit $B^-H/H$, so $C_\lambda\cong\P^1$ has two points which are in the boundary $W\setminus B^-H/H$. These two points are precisely the $\lambda$-invariant points, which are the base points of the $G$-orbits corresponding to the coloured ray generated by $v_0$ and the maximal coloured cone which contains $-v_0$. Since the former point is the only intersection of $C_\lambda$ with $D_0$, we obtain 
		\begin{align*}
			D_0\cdot C_\lambda=\varphi_{D_0}(v_0)=\frac{1}{c_0}\varphi_{D_0}(w_0)=\frac{1}{c_0}\leq 1
		\end{align*}
		where $\varphi_{D_0}:N_\R\to \R$ is the piecewise linear function associated to $D_0$. Since $\frac{1}{a_0}D_0$ and $\frac{1}{a_i}D_i$ are linearly equivalent for all $i$, we see that 
		\begin{align*}
			D_i\cdot C_\lambda = \frac{a_i}{a_0}D_0\cdot C_\lambda \leq 1 \qquad\forall i.
		\end{align*} 
		
		To start the induction argument, suppose that $n=1$. Take $C:=C_\lambda$ as above. Note that $D\cdot C\leq 1$ is satisfied for all $D\in\calD(W)$, so we just need to check that a $G$-translate of $C$ passes through any point in $W$. This holds because $C$ contains the base points of all three $G$-orbits in $W$. 
		
		Now suppose that $n>1$. To start, we assume that $p$ is in the open $G$-orbit $G/H\subseteq W$, or that $p$ is in the $G$-orbit $\calO_0$ corresponding the the coloured ray generated by $v_0$. Again, take $C:=C_\lambda$. Note that $D\cdot C\leq 1$ is satisfied for all $D\in\calD(W)$. As above, since $C$ contains the base points of $G/H$ and $\calO_0$, we can obtain a $G$-translate of $C$ which contains $p$. 
		
		Now we finish the proof by handling the case where $p$ is not in $G/H$ or $\calO_0$. There exists a $G$-orbit closure $Z=\ol{\calO}\subseteq W$, where $\calO$ is the $G$-orbit corresponding to a coloured ray generated by $v_j$ (with $j\neq 0$), such that $p\in Z$. Now $Z$ is as in Lemma \ref{lemma:Picard number 1 horospherical varieties}(2) and has coloured fan $\Sigma^c(Z)$ on the coloured lattice $N':=N/\Z v_j$. Note that $0<\rank(N')<n$, so the induction hypothesis implies that there exists a curve $C\subseteq Z$ which is the translate of the closure of a one-parameter subgroup curve by a $G(k)$-point. By construction, we have $D'\cdot C\leq 1$ for all $D'\in\calD(Z)$. Similarly to Section \ref{subsec:divisor-curve pairing on orbit closures}, we can write $\calD(Z)=\{D_i':0\leq i\leq n,~ i\neq j\}$ where $D_i'\subseteq Z$ is the prime divisor corresponding to the image of $w_i$ under the quotient $N\to N'$. 
		
		By Lemma \ref{lemma:divisor-curve pairing on orbit closure}, we have
		\begin{align*}
			D_i\cdot C \leq D_i'\cdot C \leq 1 \qquad\forall i\neq j.
		\end{align*}
		Lastly, we have
		\begin{align*}
			D_j\cdot C &= \frac{a_j}{a_0} D_0\cdot C \leq D_0'\cdot C \leq 1. 
			\qedhere
		\end{align*}
	\end{proof}
\end{lemma}

The above lemma implies that, for non-toric horospherical varieties as in Lemma \ref{lemma:Picard number 1 horospherical varieties}(2), there exists a curve satisfying Proposition \ref{prop:Picard 1 - P1 result}(1). The following result shows that this curve satisfies Proposition \ref{prop:Picard 1 - P1 result}(2) and (3). Therefore, the following result concludes the proof of Proposition \ref{prop:Picard 1 - P1 result}, which concludes the proof of Proposition \ref{prop:Picard 1 - full}. 

\begin{lemma}\label{lemma:bound on anticanonical divisor}
	Let $W$ be a non-toric horospherical $G/H$-variety as in Lemma \ref{lemma:Picard number 1 horospherical varieties}(2). Let $C$ be a curve satisfying the conditions in Lemma \ref{lemma:divisor-curve bound}. Then $-K_W\cdot C\leq \dim(W)$.
	
	\begin{proof}
		Let $\calD^G(W)\subseteq \calD(W)$ denote the set of $G$-invariant prime divisors in $W$. By \cite[Theorem 6.1.2]{perrin2018sanya}, we have $-K_W=\sum_{D\in\calD(W)} a_D D$ where $a_D=1$ if $D\in\calD^G(W)$, and if $D$ is a colour divisor then we have
		\begin{align*}
			a_D = \sum_{\beta\in R^+\setminus R_I} \langle \beta,\alpha^\vee\rangle.
		\end{align*}
		Therefore, Proposition \ref{prop:root inequality} implies
		\begin{align*}
			-K_W\cdot C &= \sum_{D\in\calD^G(W)} D\cdot C + \sum_{\alpha\in\calC} \sum_{\beta\in R^+\setminus R_I} \langle \beta,\alpha^\vee\rangle D_\alpha\cdot C
			\\&\leq \#\calD^G(W) + (\#\calC + \#(R^+\setminus R_I))
			\\&= (\#\calD^G(W)+\#\calC) + \#(R^+\setminus R_I)
			\\&= (\rank(N)+1) + \dim(G/P)
			\\&= \dim(W)+1.
		\end{align*}
		Note that we used $\#\calD^G(W)+\#\calC=\rank(N)+1$, which holds because both sides of the equality are the number of coloured rays in $\Sigma^c(W)$; and we used \cite[Theorem 6.6]{knop1991luna} for the final equality. 
			
		Lastly, since $G/P$ is not a product of projective spaces (since $W$ is not toric; see \cite[Proposition 3.3.6]{monahan2023overview}), \eqref{eq:root inequality} (from Proposition \ref{prop:root inequality}) is a strict inequality of integers, so the above inequality can be strengthened to $-K_W\cdot C\leq \dim(W)$. 
		\end{proof}
\end{lemma}

\section{Proof of Proposition \ref{prop:MMP final step}}\label{sec:proof of proposition MMP final step}

In this section, we prove Proposition \ref{prop:MMP final step}, and in turn, Theorem \ref{thm:main theorem with Q-factorial reduction} (and Theorem \ref{thm:main theorem}). Recall that Proposition \ref{prop:MMP final step} assumes the point $p$ is in the exceptional locus of an elementary MMP step. The following Lemma shows that the fibre $F$ containing $p$ has Picard number $1$. 

\begin{lemma}\label{lemma:fibres have Picard number 1}
	Suppose that $X$ is a projective $\Q$-factorial horospherical $G/H$-variety. Let $\pi:=\cont_\calR:X\to X_\calR$ be an MMP contraction of an extremal ray $\calR\subseteq\NE(X)$, let $x\in X_\calR$ be a point in the image of the exceptional locus $\Exc(\pi)$, and let $F$ denote the reduction of the fibre $\pi^{-1}(x)$. Then $F$ is a projective $\Q$-factorial horospherical variety with Picard number $1$.
	
	\begin{proof}
		If $\pi$ is a Mori fibre space, then the result follows from \cite[Theorem 4.10]{pasquier2015approach}. Now suppose that $\pi$ is not a Mori fibre space; in particular, $\pi$ is birational. As in \cite[Section 4.6]{pasquier2015approach}, we know that $F$ is a projective $\Q$-factorial horospherical variety, so it remains to show that $F$ has Picard number $1$. Set $n:=\rank(N)$. Let $\Sigma_0^c$ denote the coloured fan associated to $F$, which lives on a coloured lattice $N_0$ with universal colour set $\calC_0$. 
		
		First, suppose that $\pi$ is a wall curve contraction. Note that we necessarily have $n>0$. By Section \ref{subsec:wall curve contractions} and \cite[Lemma 15.4.2]{cox2011toric}, the underlying fan $\Sigma_0$ is the fan of a projective $\Q$-factorial toric variety with Picard number $1$, i.e. $\Sigma_0$ is a complete fan with exactly $n+1$ rays. Since $\calF(\Sigma_\calR^c)$ is inherited from $\calF(\Sigma^c)$, we see that $\calF(\Sigma_0^c)=\calC_0$. By Lemma \ref{lemma:Picard number 1 horospherical varieties}, $F$ has Picard number $1$. 
		
		Lastly, suppose that $\pi$ is a colour curve contraction. Then, by Section \ref{subsec:colour curve contractions}, we have $n=0$ and $\calC_0$ consists of one colour. Hence, $\Sigma_0^c$ is a coloured fan of the form (1) in Lemma \ref{lemma:Picard number 1 horospherical varieties}, so $F$ has Picard number $1$. 
	\end{proof}
\end{lemma}

With the above lemma, we can use the results in Section \ref{sec:results on Picard number 1 horospherical varieties} to say that Theorem \ref{thm:main theorem with Q-factorial reduction} holds for the fibre $F$ containing $p$. The following lemma provides an inequality which helps us reduce the general problem on $X$ to $F$. 

\begin{lemma}\label{lemma:canonical divisor of fibre inequality}
	Let $X$ be a projective terminal $\Q$-factorial horospherical $G/H$-variety, and let $\pi:X\to Y$ be an MMP contraction. Suppose that $C\subseteq X$ is a curve contracted by $\pi$. If $F$ is the reduction of the fibre of $\pi$ containing $C$, then $-K_X\cdot C\leq -K_F\cdot C$. 
	
	\begin{proof}
		Let $y:=\pi(F)$. Consider the unique $G$-orbit in $Y$ containing $y$, and let $V$ denote the closure of this orbit. Since the fibres of $\pi$ are irreducible, $\pi^{-1}(V)$ is irreducible, so its reduction $Z$ is a $G$-orbit closure in $X$. Let $\tau^c\in\Sigma^c(X)$ denote the coloured cone corresponding to $Z$. Since $F$ is positive-dimensional and is a general fibre of $\res{\pi}{Z}:Z\to V$, we see that $Z$ is contained in the exceptional locus $\Exc(\pi)$. Furthermore, since $F$ is a general fibre of $\res{\pi}{Z}$, we have $\res{K_Z}{F}=K_F$, so it suffices to show that $-K_X\cdot C\leq -K_Z\cdot C$. We use the notation of Section \ref{subsec:divisor-curve pairing on orbit closures} for the $B^-$-invariant prime divisors of $X$ and $Z$.
		
		If $C$ is a wall curve then, using the proof of \cite[Lemma 7.1]{satriano2021approximating}, we have $D_i\cdot C\leq 0$ for all $0\leq i\leq a$. On the other hand, if $C$ is a colour curve corresponding to $\gamma\in\calC$, then $\tau$ is the smallest cone in $\Sigma(X)$ containing $u_\gamma$, and $D_i\cdot C\leq 0$ for all $0\leq i\leq a$ by \eqref{eq:colour curve class}. Therefore, in any case, we have $D_i\cdot C\leq 0$ for all $0\leq i\leq a$, and it is easy to check that $D_i\cdot C=0$ for all $b+1\leq i\leq l$. By \cite[Theorem 6.1.2]{perrin2018sanya}, we can write 
		\begin{align*}
			-K_X = \sum_{D\in\calD(X)} a(D) D \quad\text{and}\quad -K_Z = \sum_{D'\in\calD(Z)} a(D') D'
		\end{align*}
		for some $a(D),a(D')\in\Z_{>0}$, where $a(D)=a(D')=1$ for all $G$-invariant prime divisors $D\subseteq X$, $D'\subseteq Z$. By Lemma \ref{lemma:inequality for coefficients of flag canonical}, it follows that $a(D_i)\leq a(D_i')$ and $a(D_\alpha^X)\leq a(D_\alpha^Z)$ for all $a+1\leq i\leq b$ and all $\alpha\in\calC\setminus\calF(\Star(\tau^c))$. Putting all this and Lemma \ref{lemma:divisor-curve pairing on orbit closure} together, we obtain
		\begin{align*}
			-K_X\cdot C = \sum_{D\in\calD(X)} a(D)D\cdot C &\leq \sum_{i=a+1}^b a(D_i)D_i\cdot C + \sum_{\alpha\in\calC\setminus\calF(\Star(\tau^c))} a(D_\alpha^X) D_\alpha^X\cdot C 
			\\&\leq \sum_{i=a+1}^b a(D_i') D_i'\cdot C + \sum_{\alpha\in\calC\setminus\calF(\Star(\tau^c))} a(D_\alpha^Z) D_\alpha^Z\cdot C 
			\\&=\sum_{D'\in\calD(Z)} a(D')D'\cdot C = -K_Z\cdot C.
			\qedhere
		\end{align*}
	\end{proof}
\end{lemma}

In the case where $X$ is isomorphic to projective space, Theorem \ref{thm:main theorem with Q-factorial reduction} follows from \cite[Theorem 2.6]{mckinnon2007conjecture}. Therefore, the following result concludes the proof of Proposition \ref{prop:MMP final step}. 

\begin{proposition}\label{prop:proving of MMP final step}
	Let $X$ be a projective terminal $\Q$-factorial horospherical variety, let $\pi:X\to Y$ be an MMP contraction corresponding to an extremal ray $\calR$, let $D$ be a nef $\Q$-divisor on $X$ such that $D\in\calR^\perp$, and let $p\in X(k)\cap\Exc(\pi)$ be a canonically bounded point. If $X$ is not isomorphic to projective space, then there exists an irreducible rational curve $C$ through $p$ such that $C$ is unibranch at $p$, $-K_X\cdot C\leq \dim(X)$, and $\alpha_{p,C}(D)\leq \alpha_{p,\{x_i\}}(D)$ for all Zariski dense sequences $\{x_i\}\subseteq X(k)$. 
	
	\begin{proof}
		This follows from an almost identical proof of \cite[Proposition 7.2]{satriano2021approximating} by using Lemma \ref{lemma:canonical divisor of fibre inequality} (the analogue of \cite[Lemma 7.1]{satriano2021approximating}), Lemma \ref{lemma:fibres have Picard number 1}, and Proposition \ref{prop:Picard 1 - full} (the analogue of \cite[Proposition 6.1]{satriano2021approximating}). 
	\end{proof}
\end{proposition}

\printbibliography 
	
\end{document}